\documentclass[a4paper,12pt]{article}

\usepackage[utf8x]{inputenc}
\usepackage{amsmath}
\usepackage{amsfonts}
\usepackage{amssymb}
\usepackage{graphicx}
\usepackage{amsthm}
\usepackage{enumerate}

\newtheorem{theo}{Theorem}

\newtheorem{conj}{Conjecture}
\newtheorem{prop}{Proposition}

\def\x{{\sf x}}
\def\ox{\overline{\sf x}}
\def\y{{\sf y}}
\def\oy{\overline{\sf y}}
\def\z{{\sf z}}
\def\oz{\overline{\sf z}}
\def\C{{\sf C}}
\def\t{{\sf true}}
\def\f{{\sf false}}

\title{On Frank's conjecture on $k$-connected orientations}

\author{Olivier Durand de Gevigney\thanks{Laboratoire G-SCOP, CNRS,
Grenoble-INP, UJF, France} \thanks{This research was conducted while the author
was visiting the University of Waterloo. The author was supported by a grant
Explora Doc from Rh\^one-Alpes and NSERC grant No. OGP0138432.}}

\begin{document}

\maketitle

\begin{abstract}
  We disprove a conjecture of Frank \cite{Frank1995} stating that each weakly
  $2k$-connected graph has a $k$-vertex-connected orientation. For $k \geq 3$,
  we also prove that the problem of deciding whether a graph has a
  $k$-vertex-connected orientation is NP-complete.
\end{abstract}

\section*{Introduction}

An \emph{orientation} of an undirected graph $G$ is a digraph obtained from
$G$ by substituting an arc $uv$ or $vu$ for every edge $uv$ in $G$. We are
interested in characterizing graphs admitting an orientation that satisfies
connectivity properties.
Robbins \cite{Robbins1939} proved that a graph $G$ admits a strongly connected
orientation if and only if $G$ is $2$-edge-connected. The following extension to
higher connectivity follows from of a result of Nash-Williams
\cite{NashWilliams1960}: a graph $G$ admits a $k$-arc-connected orientation if
and only if $G$ is $2k$-edge-connected.

Little is known about vertex-connected orientations. Thomassen
\cite{Thomassen1989} conjectured that if a graph has sufficiently high
vertex-connectivity then it admits a $k$-vertex-connected orientation.
\begin{conj}[Thomassen \cite{Thomassen1989}]
  For every positive integer $k$ there exists an integer $f(k)$ such that every
  $f(k)$-connected graph admits a $k$-connected orientation.
\end{conj}
The case $k=2$ has been proved by Jord\'an \cite{Jordan2005} by showing $f(2)
\leq 18$. Recently, it was shown in \cite{CheriyanDdGSzigeti} that $f(2) \leq
14$. However, the conjecture of Thomassen remains open for $k \geq 3$. 

A graph $G=(V,E)$ is called \emph{weakly $2k$-connected} if $|V|>k$ and for all
$U \subseteq V$ and $F \subseteq E$ such that $2|U|+|F|<2k$, the graph
$G-U-E$ is connected. It is easy to see that any graph
admitting a $k$-connected orientation is weakly $2k$-connected. Note that
checking the weak $2k$-connectivity of a graph can be done in polynomial time
using a variation of the Max-flow Min-cut algorithm \cite{FordFulkerson1962}.
Frank \cite{Frank1995} conjectured that this connectivity condition
characterizes graphs admitting a $k$-connected orientation.
\begin{conj}[Frank \cite{Frank1995}]\label{FrankC}
  A graph $G$ admits a $k$-connected orientation if and only if $G$ is
  weakly $2k$-connected.
\end{conj}
Berg and Jord\'an \cite{BergJordan2006} proved this conjecture for the special
case of Eulerian graphs and $k=2$. For a short proof of this result, see
\cite{KiralySzigeti2006}. In this article we disprove this conjecture for $k
\geq 3$. For instance, the graph $G_3$ in Figure \ref{fg:kEqual} is a
counterexample for $k=3$. We also prove that deciding whether a given graph has
a $k$-connected orientation is NP-complete for $k \geq 3$. Both these results
hold also for the special case of Eulerian graphs. Hence assuming $P \neq NP$,
there is no good characterisation of graphs admitting a $k$-connected
orientation for $k\geq 3$. We mention that counterexamples can easily be derived
from our NP-completeness proof, but we give simple self-contained
counterexamples. Furthermore, the gadgets used in the NP-completeness proof are
based on properties used in our first counterexample.

This paper is organized as follows. In Section \ref{pre} we establish the
necessary definitions and some elementary results. In Section \ref{ce} we
disprove Conjecture \ref{FrankC} for $k \geq 3$. For $k \geq 4$, we provide
Eulerian counterexamples. 
In Section \ref{npc}, for $k \geq 3$, we reduce the problem of
\textsc{Not-All-Equal $3$-Sat} to the problem of finding a $k$-connected
orientation of a graph. This reduction leads to an Eulerian counterexample of
Conjecture \ref{FrankC} for $k=3$.

\section{Preliminaries}\label{pre}

Let $k$ be a positive integer and let $D=(V,A)$ be a digraph. We mention that
digraphs may have multiple arcs. In $D$ the indegree (respectively, the outdegree) of a
vertex $v$ is denoted by $\rho_D(v)$ (respectively,  by $\delta_D(v)$). 
The pair $u,v \in V$ is called \emph{strongly connected} if there exist a dipath
from $u$ to $v$ and a dipath from $v$ to $u$. The digraph $D$ is called strongly
connected if every pair of vertices is strongly connected.
The pair $u,v \in V$ is called \emph{$k$-connected} if, for all $U \subseteq V
\setminus \{u,v\}$ such that $|U|<k$, $u$ and $v$ are strongly connected in the
digraph $D-U$. A set of vertices is called $k$-connected if
every pair of vertices contained in this set is $k$-connected. The digraph $D$
is called $k$-connected if $|V|>k$ and $V$ is $k$-connected.

Let $G=(V,E)$ be a graph. We mention that graphs may have multiple edges. In $G$
the degree of a vertex $v$ is denoted by $d_G(v)$ and the number of edges
joining $v$ and a subset $U$ of $V-v$  is denoted by $d_G(v,U)$. 
The pair $u,v \in V$ is called \emph{connected} if there is a path joining $u$
and $v$. 
The pair $u,v \in V$ is called \emph{weakly $2k$-connected} if, for all $U
\subseteq V \setminus \{u,v\}$ and $F \subseteq E$ such that $2|U|+|F|<2k$, $u$
and $v$ are connected in the graph $G-U-F$. A set of
vertices is called weakly $2k$-connected if every pair of vertices contained in
this set is weakly $2k$-connected. So $G$ is weakly $2k$-connected if $|V|>k$
and $V$ is weakly $2k$-connected.\\

The constructions in this paper are based on the following facts.
\begin{prop}\label{elementary}
  Let $G=(V,E)$ be a graph admitting a $k$-connected orientation $D$.  Let $v$
  be a vertex of degree $2k$ and $u \neq v$ be a vertex such that $d_G(u,v)=2$.
  Then $\rho_D(v) = \delta_D(v) = k$ and the two parallel edges between $u$ and
  $v$ have opposite directions in $D$.
\end{prop}
\begin{proof}
  By $k$-connectivity of $D$, the indegree (respectively, the outdegree) of $v$ is at
  least $k$. Hence, since $2k = d_G(v) = \rho_D(v) + \delta_D(v)$ we have
  $\rho_D(v)=\delta_D(v)=k.$ Now suppose for a contradiction that the two
  parallel edges between $u$ and $v$ have the same direction, say from $u$ to
  $v$. Then the set of vertices that have an outgoing arc to $v$ is smaller than
  $k$ and deleting this set results in a digraph that is not strongly connected, a
  contradiction.
\end{proof}

For $U \subset V$, a pair of dipaths of $D$ (respectively, paths of $G$) is called
$U$-disjoint if each vertex of $U$ is contained in at most one dipath (respectively,
path).
Let $X$ and $Y$ be two disjoint vertex sets. A \emph{$k$-difan from $X$ to $Y$}
(respectively, a \emph{$k$-fan joining $X$ and $Y$}) is a set of $k$ pairwise $U$-disjoint
dipaths from $X$ to $Y$ (respectively, paths joining $X$ and $Y$) where $U$ is defined
by $U=V\setminus (X \cup Y)$ if $|X|=|Y|=1$, $U=V \setminus X$ if $|X|=1$ and
$|Y|>1$, $U=V \setminus Y$ if $|Y|=1$ and $|X|>1$, $U=V$ if $|X|>1$ and $|Y|>1$.

By Menger's theorem \cite{Menger1927}, a pair $u,v$ of vertices of $D$ is
$k$-connected if and only if there exist a $k$-difan from $u$ to $v$ and a
$k$-difan from $v$ to $u$. 
Let $X$ be a $k$-connected set of at least $k$ vertices and let $v$ be a vertex in
$V\setminus X$ such that there exist a $k$-difan from $X$ to $v$ and a $k$-difan
from $v$ to $X$; then, it is easy to prove that $X \cup v$ is $k$-connected.
%It is easy to prove that for a $k$-connected set $X$ of at least $k$ vertices
%and $v \in V \setminus X$ if there exist a $k$-difan from $X$ to $v$ and a
%$k$-difan from $v$ to $X$ then $X \cup v$ is $k$-connected.

Kaneko and Ota \cite{Kaneko_Ota2000} showed that a pair $u,v$ of vertices of $G$
is weakly $2k$-connected if and only if there exist $2$ edge-disjoint $k$-fans
joining $u$ and $v$. 
Let $X$ be a weakly $2k$-connected set of at least $k$ vertices and let $v$ be a
vertex in $V \setminus X$ such that there exist $2$ edge-disjoint $k$-fans
joining $v$ and $X$; then, it is easy to prove that $X \cup v$ is weakly $2k$-connected.
%It is easy to prove that for a weakly $2k$-connected set $X$ of at least $k$
%vertices and $v \in V \setminus X$ if there exist $2$ edge-disjoint $k$-fans
%joining $v$ and $X$ then $X \cup v$ is weakly $2k$-connected. 
Let $X$ and $Y$ be two disjoint weakly $2k$-connected sets each of at least $k$
vertices such that there exist $2$ edge-disjoint $k$-fans joining $X$ and $Y$; then,
it is easy to prove that $X \cup Y$ is weakly $2k$-connected.
%It is also easy to prove that for two disjoint weakly $2k$-connected sets $X,Y$
%each of at least $k$ vertices if there exist $2$ edge-disjoint $k$-fans joining
%$X$ and $Y$ then $X \cup Y$ is weakly $2k$-connected. 

\section{Counterexamples} \label{ce}

We first disprove Conjecture \ref{FrankC} for $k=3$ and then extend the idea
of the proof to higher connectivity. We recall that $G_3$ is the graph defined
in Figure \ref{fg:kEqual}.

\begin{figure}[ht]
  \centering
  \includegraphics{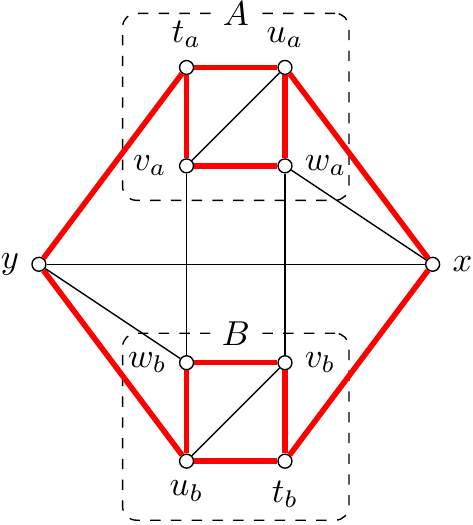}
  \caption{$G_3$ every thick and red edge represents a pair of
  parallel edges and black edges represent simple edges.}
  \label{fg:kEqual}
\end{figure}
\begin{prop}
  The graph $G_3$ is weakly $6$-connected and has no $3$-connected orientation.
\end{prop}

\begin{proof}
  First we show that $G_3$ is weakly $6$-connected.  Observe that there exist
  $2$ edge-disjoint $3$-fans joining any pair of vertices in $A \setminus w_a$.  Then, note
  that there exist $2$ edge-disjoint $3$-fans joining $w_a$ and $A \setminus w_a$. Hence $A$
  is weakly $6$-connected. Symmetrically $B$ is also weakly $6$-connected. There
  exist $2$ edge-disjoint $3$-fans joining $A$ and $B$ so $A \cup B$ is weakly
  $6$-connected.  There exists $2$ edge-disjoint $3$-fans joining $x$ (respectively,
  $y$) and $A\cup B$. It follows that $G_3$ is weakly $6$-connected.

  Suppose for a contradiction that $G_3$ has a $3$-connected orientation $D$.
  Note that every pair of parallel edges is incident to a vertex of degree $6$
  and the maximal edge multiplicity is $2.$ Hence, by Proposition
  \ref{elementary}, the two edges in every parallel pair have opposite
  directions in $D$.  Thus, in $D$ the orientation of the edges of the path $u_a
  v_a w_b y x w_a v_b u_b$ results in a directed path from $u_a$ to $u_b$ or
  from $u_b$ to $u_a$. In particular both $v_a w_b$ and $v_b w_a$ are directed
  from $A$ to $B$ or from $B$ to $A$.  In both cases $D-\{x,y\}$ is not strongly
  connected, a contradiction.
\end{proof}

We mention that $G_3$ is not a minimal counterexample. Indeed the graph $H_3$
obtained from $G_3$ by deleting the two vertices $t_a$ and $t_b$ and adding the
new edges $u_a v_a$, $v_a y$, $y u_a$, $u_b v_b$, $v_b x$ and $x u_b$ is weakly
$6$-connected but has no $3$-connected orientation. (Suppose that $H_3$ has a
$3$-connected orientation $D$. Then, by Proposition \ref{elementary}, in $D$ the
orientation of the edges of the two triangles $v_a y w_b$ and $v_b x w_a$
results in circuits. Considering the cut $\{x,y\}$, we see that those circuits
must be either both clockwise or both counterclockwise, say clockwise. Hence, by
Proposition \ref{elementary}, in $D$ the orientation of the path $u_a y x u_b$
results in a dipath from $u_a$ to $u_b$ or from $u_b$ to $u_a$. In the first
case $D-\{y,v_b\}$ is not strongly connected, in the other case $D-\{x,v_a\}$ is
not strongly connected.)

We now extend this construction to higher connectivity. Let $k \geq 4$ be
an integer. We define the graph $G_k=(V,E)$ as follows (see Figure
\ref{fg:kGeq4}).
\begin{figure}[ht]
  \centering
  \includegraphics{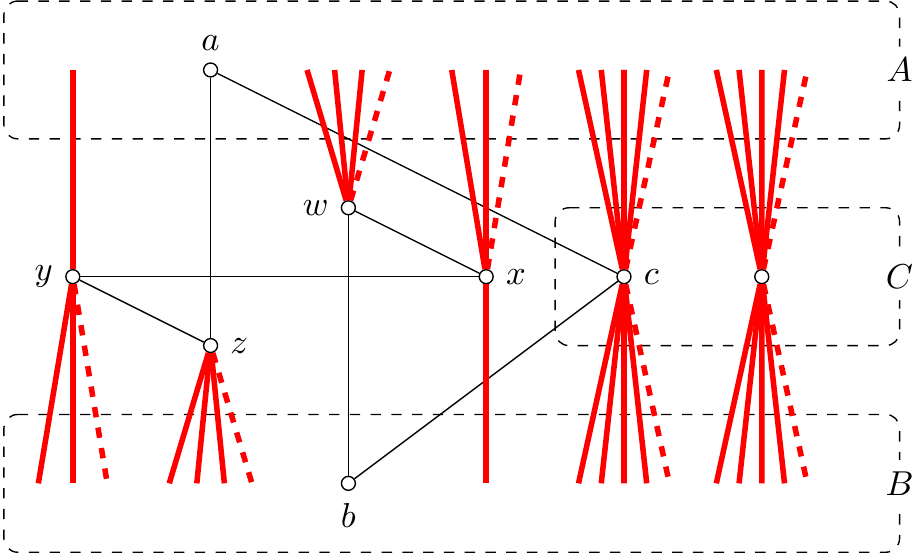}
  \caption{$G_k$ every thick and red edge represents a pair of
  parallel edges and black edges represent simple edges.}
  \label{fg:kGeq4}
\end{figure}
Let $n \geq k^2$ be an odd integer. The vertex set $V$ is the union of the
pairwise disjoint sets $A$, $B$, $C$ and $\{w,x,y,z\}$ where $|A|=|B|=n$ and
$|C|=k-3$. Now we add simple edges such that each of $A$ and $B$ induces a
complete simple graph. Choose arbitrarily one vertex from each of $A$, $B$ and
$C$, say $a \in A$, $b \in B$ and $c\in C$ and add the cycle $azyxwbc$. By the
choice of $n$, we can now add pairs of parallel edges between vertices in $A
\cup B \setminus \{a,b\}$ and $C \cup \{w,x,y,z\}$ such that
\begin{eqnarray*}
  \textrm{each vertex of} & A \cup B & \textrm{ is incident to at most one pair of parallel
  edges,}\\ 
  d_{G_k}(v,A)&=&d_{G_k}(v,B)=2\lceil \frac{k}{2} \rceil \textrm{ for all $v \in
  C-c$,}\\
d_{G_k}(c,A)&=&d_{G_k}(c,B)=2\lceil \frac{k}{2} \rceil+1,\\
d_{G_k}(w,A)&=&d_{G_k}(z,B)=2k-2,\\
d_{G_k}(y,A)&=&d_{G_k}(x,B)=2 \textrm{ and}\\
d_{G_k}(x,A)&=&d_{G_k}(y,B)=2k-4.
\end{eqnarray*}
\begin{prop}
  Let $k\geq 4$ be an integer. The graph $G_k$ is Eulerian, weakly
  $2k$-connected and has no $k$-connected orientation.
\end{prop}
\begin{proof}
Since $n$ is odd, both of the complete graphs induced by $A$ and $B$ are
Eulerian. Hence $G_k$, which is obtained from those graphs by adding a cycle and
parallel edges, is Eulerian.
Since $k \geq 4$, $n \geq k^2 \geq 2k+2$ thus both of the complete graphs
induced by $A$ and $B$ are weakly $2k$-connected. Note that there exist $2$
edge-disjoint $k$-fans joining $A$ and $B$ (one uses $C \cup \{w,x,y\}$
the other one uses $C \cup \{x,y,z\}$), thus
$A \cup B$ is weakly $2k$-connected. Note also that, for any
vertex $v \in C \cup \{w,x,y,z\}$, there exist $2$ edge-disjoint $k$-fans
joining $v$ and $A \cup B.$ Hence, $G_k$ is weakly $2k$-connected.

Suppose for a contradiction that $G_k$ has a $k$-connected orientation $D$.
Since $d_{G_k}(w)=d_{G_k}(x)=d_{G_k}(y)=d_{G_k}(z)=2k$ and by Proposition
\ref{elementary}, the orientation of the set of simple edges of the path
$azyxwb$ results in the dipath $azyxwb$ or the dipath $bwxyza$. In
both cases, $D-(C \cup \{x,y\})$ is not strongly connected, a contradiction.
\end{proof}

Note that with a slightly more elaborate construction we can obtain a
counterexample such that $|V| = O(k)$.

\section{NP-completeness}\label{npc}

In this section we prove the following result.
\begin{theo}\label{NPCtheo}
  Let $k \geq 3$ be an integer. Deciding whether a graph has a $k$-connected
  orientation is NP-complete. This holds also for Eulerian graphs.
\end{theo}

A \emph{reorientation} of a digraph $D$ is a digraph obtained from $D$ by
reversing a subset of arcs. Obviously, the problem of finding a $k$-connected
orientation of a graph and the problem of finding a $k$-connected reorientation
of a digraph are equivalent. For convenience we prove the NP-completeness of the
second problem by giving a reduction from the problem of \textsc{Not-All-Equal
$3$-Sat} which is known to be NP-complete \cite{Schaefer1978}.\\

Let $\Pi$ be an instance of \textsc{Not-All-Equal $3$-Sat} and let $k \geq 3$ be
an integer. We define a directed graph $D_k=D_k(\Pi)=(V,A)$ such that there exists
a $k$-connected reorientation of $D_k$ if an only if there is an assignment of the
variables which satisfies $\Pi$.

The construction of $D_k$ associates to each variable $\x$ a circuit $\Delta_\x$ and
to each pair $(\C,\x)$ where $\x$ is a variable that appears in the clause $\C$
a special arc $e^\C_\x$ (see Figure \ref{fg:GeneralGadget}). A reorientation of
$D_k$ is called \emph{consistent} if the orientation of parallel edges is
preserved and, for each variable $\x$, the orientations of the special arcs of
type $e^\C_\x$ and the circuit $\Delta_\x$ are either all preserved or all
reversed.  A consistent reorientation of $D_k$ defines a \emph{natural
assignment} of the variables in which a variable $\x$ receives value $\t$ if
$\Delta_\x$ is preserved and $\f$ if $\Delta_\x$ is reversed. We define
reciprocally a \emph{natural consistent reorientation} from an assignment of the
variables.

\begin{figure}[ht]
  \centering
  \includegraphics{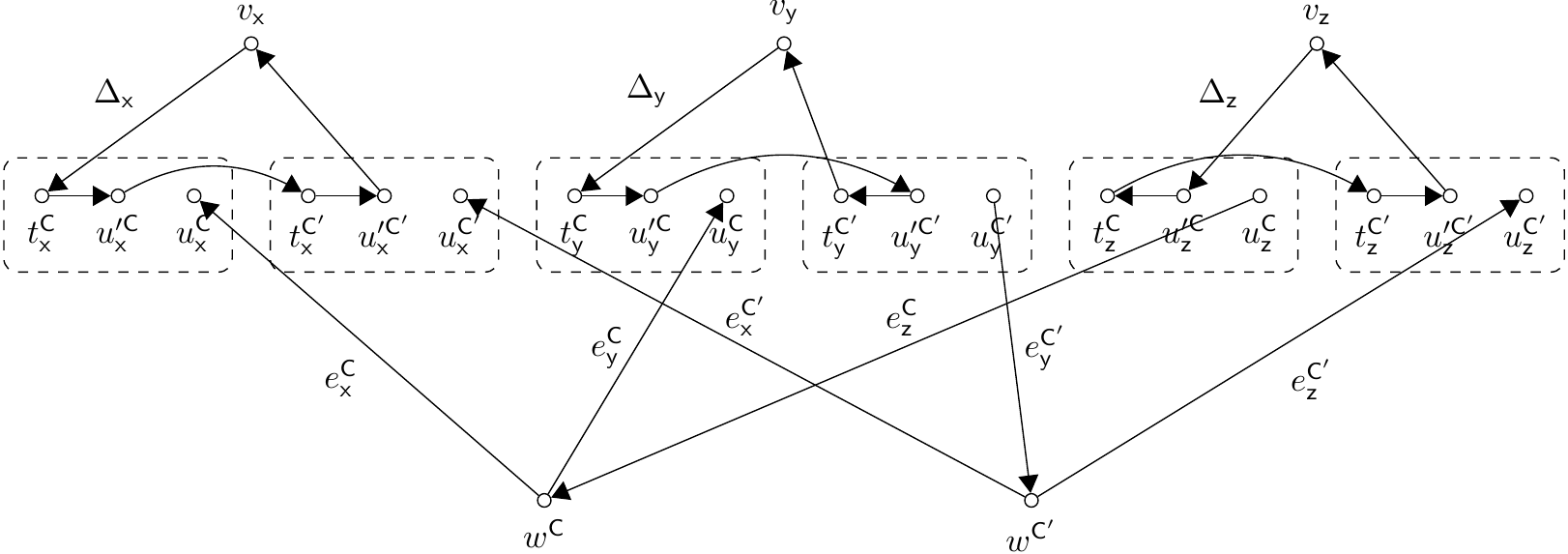}
  \caption{Representation of the circuits and the special arcs of $D_3(\Pi)$
  where $\Pi$ is composed of the clauses $\C=(\x,\y,\oz)$ and $\C'=(\x,\oy,\z)$. The dashed boxes represent the
  clause-variable gadgets.}
  \label{fg:GeneralGadget}
\end{figure}

For each clause $\C$ we construct a $\C$-gadget (see Figure \ref{fg:ClauseGadget})
that uses the special arcs associated to $\C$.  The purpose of the $\C$-gadgets is
to obtain the following property.
\begin{prop}\label{clauseProp}
  An assignment of the variables satisfies $\Pi$ if and only if it defines a
  natural consistent $k$-connected reorientation.
\end{prop}
For each pair $(\C,\x)$ where $\C$ is a clause and $\x$ is a
variable that appears in $\C$ we define a $(\C,\x)$-gadget (see Figure
\ref{fg:VariableGadget}) which links the orientation of $\Delta_\x$ to the
orientation of $e^\C_\x$. We will prove the following fact.
\begin{prop}\label{variableProp}
  If there exists a $k$-connected reorientation of $D_k$ then there exists a
  consistent $k$-connected reorientation of $D_k$.
\end{prop}

\begin{figure}[ht]
  \begin{minipage}[b]{0.45\linewidth}
    \centering
    \includegraphics{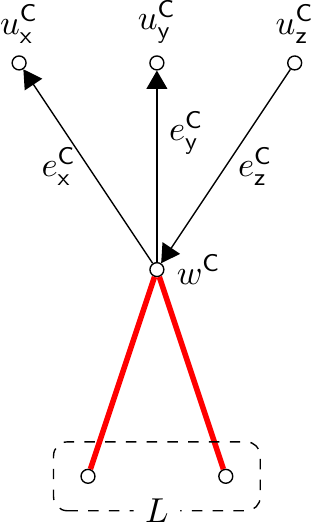}
    \caption{A clause gadget for $k=3$ and $\C=(\x,\y,\oz)$. Each red and thick
    edge represents a pair of parallel arcs in opposite directions.}
    \label{fg:ClauseGadget}
  \end{minipage}
  \hspace{0.08\linewidth}
  \begin{minipage}[b]{0.45\linewidth}
    \centering
    \includegraphics{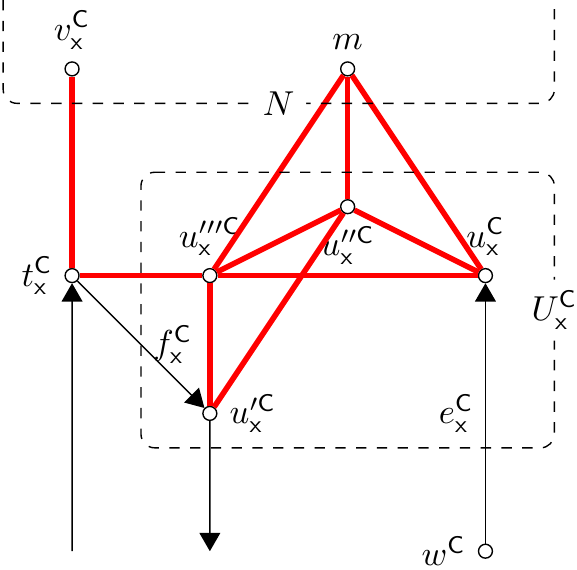}
    \caption{A $(\C,\x)$-gadget for $k=3$ and $\x \in \C$. Each red and thick
    edge represents a pair of parallel arcs in opposite directions.}
    \label{fg:VariableGadget}
  \end{minipage}
\end{figure}

Let $L$ be a set of $k-1$ vertices. We construct a clause gadget as follows.
For a clause $\C$ composed of the variables $\x,\y$ and $\z$ we add
the vertices $w^\C,u^\C_\x,u^\C_y,u^\C_\z$. We add arcs such that $L \cup w^\C$
induces a complete digraph. We add the special arc $w^\C u^\C_\x$ if $\x \in \C$
and the special arc $u^\C_\x w^\C$ if $\ox \in \C$. This special arc is denoted
by $e^\C_\x$. We define similarly the special arcs $e^\C_\y$ and $e^\C_\z$. This
ends the construction of the $\C$-gadget. Let $W$ denote the set of all vertices
of type $w^\C$.

Let $M$ be a set of $k-2$ new vertices and choose arbitrarily one vertex $m \in
M$. For each pair $(\C,\x)$ where $\C$ is a clause and $\x$ is a variable that
appears in $\C$ we add the new vertices $t^\C_\x, u'^\C_\x, u''^\C_\x,
u'''^\C_\x, v^\C_\x$ and denote $U^\C_\x = \{u^\C_\x, u'^\C_\x, u''^\C_\x,
u'''^\C_\x\}$. We add arcs such that $M \cup (U^\C_\x \setminus u^\C_\x)$ induces a
complete digraph. We add pairs of parallel arcs in opposite directions between
the pairs of vertices $(v^\C_\x,t^\C_\x)$, $(t^\C_\x,u'''^\C_\x)$, 
$(u'''^\C_\x,u'^\C_\x)$, $(u'^\C_\x,u''^\C_\x)$ and all the pairs of type
$(t^\C_\x,m')$ and $(u'^\C_\x,m')$ for each $m' \in M \setminus m$. Note that, so
far, the undirected degree of $t^\C_\x$ and $u'^\C_\x$ is $2k-2$. We add an arc
$t^\C_\x u'^\C_\x$ if $\x \in \C$ and an arc $u'^\C_\x t^\C_\x$ if $\ox \in \C$.
Call this arc $f^\C_\x$. The definition of the $(\C,\x)$-gadget is concluded by
the following definition of the circuit $\Delta_\x$.

For each variable $\x$ define a new vertex $v_\x$ and add arcs such that $v_\x$
and the set of vertices of type $t^\C_\x$ and $u'^\C_\x$ induce a circuit
$\Delta_\x$ that traverses (in arbitrary order) all the $(\C,\x)$-gadgets such
that $\C$ is a clause containing $\x$. In this circuit connect a
$(\C,\x)$-gadget to the next $(\C',\x)$-gadget by adding an arc leaving the head
of $f^\C_\x$ and entering the tail of $f^{\C'}_\x$ (see Figure
\ref{fg:GeneralGadget}). Note that now the undirected degree of $t^\C_\x$ and
$u'^\C_\x$ is $2k$.

We denote by $N$ the union of $L$, $M$ and all the vertices of type $v_\x$ or
$v^\C_\x$. To conclude the definition of $D_k$ we add edges such that $N$
induces a complete digraph.\\

The proof of Proposition \ref{variableProp} follows from the construction of the
$(\C,\x)$-gadgets.
\begin{proof}[Proof of Proposition \ref{variableProp}]
  Let $D'$ be a $k$-connected reorientation of $D_k$ and let $\x$ be a variable.
  Observe that all the vertices incident to $\Delta_\x$ except $v_\x$ are of degree
  $2k$ and incident to $k-1$ pairs of parallel edges.  Hence, by Proposition
  \ref{elementary}, $\Delta_\x$ is either preserved or reversed.  Let $\C$ be a
  clause in which $\x$ appears. In $D'-(M \cup t^\C_\x)$ exactly one
  arc enters $U^\C_\x$ and exactly one arc leaves $U^\C_\x$ (see Figure
  \ref{fg:VariableGadget}). One of these arcs belongs to $\Delta_\x$ and the
  other is the special arc $e^\C_\x$. Hence, by $k$-connectivity of $D'$,
  $e^\C_\x$ is reversed if and only if $\Delta_\x$ is reversed.

  If there exists a pair of parallel arcs in the same direction in $D'$ then
  reversing the orientation of one arc of this pair preserves the
  $k$-connectivity. Hence we may assume that in $D'$ the orientation of parallel
  edges is preserved.
\end{proof}

The following fact follows easily from the definition of $D_k$. We recall that
$W$ is the set of vertices of type $w^\C$.
\begin{prop}\label{almostConnectivity}
  In every consistent reorientation of $D_k$ the set $V \setminus W$ is $k$-connected.
\end{prop}
\begin{proof}
  Let $D'$ be a consistent reorientation of $D_k$.
  Clearly $N$ is $k$-connected. Let $\C$ be a clause and $\x$ be a variable that
  appears in $\C$. The circuit $C_\x$ contains a dipath from (respectively, to)
  $t^\C_\x$ to (respectively, from) $v_\x$ that is disjoint from $M \cup v^\C_\x$. Hence
  $N \cup t^\C_\x$ is $k$-connected. 
  
  We may assume without loss of generality that, in $D'-(M \cup t^\C_\x)$, the
  special arc $e^\C_\x$ enters $U^\C_\x$ and an arc of $\Delta_\x$ leaves
  $U^\C_\x$. Let $u$ be a vertex of $U^\C_\x$. Observe that there is a $k$-difan
  from $u$ to $M \cup t^\C_\x \cup v_\x$ (the dipath to $v_\x$ uses arcs of
  $\Delta_\x$). Observe that there is a $k$-difan from $M \cup t^\C_\x \cup L$
  to $u$ (the dipath from $L$ uses the arc $e^\C_\x$). Hence, since $M$ and $L$ are
  subsets of $N$, $N \cup t^\C_\x \cup U^\C_\x$ is $k$-connected and the
  proposition follows.
\end{proof}

We can now prove Proposition \ref{clauseProp}.
\begin{proof}[Proof of Proposition \ref{clauseProp}]
  Let $\Omega$ be an assignment of the variables and $D'$ the natural consistent
  reorientation of $D_k$ defined by $\Omega$.  Let $e^\C_\x$ be a special arc
  associated to a clause $\C$ and a variable $\x$. In $D'$, the arc $e^\C_\x$
  leaves $w_\C$ if and only if $\x=\t$ and $\x \in \C$ or $\x=\f$ and $\ox \in
  \C$. And, in $D'$, the arc $e^\C_\x$ enters $w_\C$ if and only if $\x=\t$ and
  $\ox \in \C$ or $\x=\f$ and $\x \in \C$.  Hence $\C$ contains a $\t$ (respectively,
  $\f$) value if and only if there exists a special arc leaving (respectively, entering)
  $w^\C$ in $D'$. Thus a clause $\C$ is satisfied by $\Omega$ if and only if
  \begin{align}
    w^\C \textrm{ is left by at least one special arc} \notag\\
    \textrm{and entered by at least one special arc.} \tag{$\star$}
  \end{align}
  Observe that, for each clause $\C$, the only arcs incident to $w_\C$ in $D'-L$
  are special.  Since $|L|=k-1$, if $D'$ is $k$-connected then ($\star$) holds
  for all clauses thus $\Omega$ satisfies $\Pi$.  Conversely, if $\Omega$
  satisfies $\Pi$ then for every clause $\C$ ($\star$) holds and $w^\C$ has at
  least $k$ out-neighbors and at least $k$ in-neighbors. Thus by Proposition
  \ref{almostConnectivity} $D'$ is $k$-connected
\end{proof}

Denote by $G'_k=G'_k(\Pi)$ the underlying undirected graph of $D_k(\Pi)$.
We can now prove the main theorem of this section.
\begin{proof}[Proof of Theorem \ref{NPCtheo}]
  By Propositions \ref{variableProp} and \ref{clauseProp}, $G'_k(\Pi)$ has a
  $k$-connected orientation if and only if there exists an assignment satisfying
  $\Pi$. Since the order of $G'(\Pi)$ is a linear function of the size of $\Pi$
  and \textsc{Not-All-Equal $3$-Sat} is NP-complete \cite{Schaefer1978} this
  proves the first part of Theorem \ref{NPCtheo}.

  Observe that in $G'_k$ the only vertices of odd degree are of type $u^\C_\x$
  and $w^\C$. Let $l$ be an arbitrary vertex of $L$. We can add a set $F$ of
  edges of type $u^\C_\x m,ml,lw^\C$ such that $G'_k+F$ is Eulerian. Observe
  that for any orientation of $F$, Propositions \ref{variableProp} and
  \ref{clauseProp} still hold for $D_k+F$. This proves the second part of
  Theorem \ref{NPCtheo}.
\end{proof}

The following fact shows that $G'_k(\Pi)$ is a counterexample to Conjecture
\ref{FrankC} if $\Pi$ is not satisfiable.
\begin{prop}
  The graph $G'_k(\Pi)$ is weakly $2k$-connected.
\end{prop}
\begin{proof}
  By Proposition \ref{almostConnectivity}, $V \setminus W$ is
  $k$-connected in $D_k$, thus $V \setminus W$ is weakly $2k$-connected in $G'_k$. Since
  there exist $2$ edge-disjoint $k$-fans from $w^\C$ to $V \setminus W$ for
  every clause $\C$, $G'_k$ is weakly $2k$-connected.
\end{proof}

We now construct an Eulerian counterexample to Conjecture \ref{FrankC} for $k=3$.
Let $\x$ be a variable and $\C=(\x,\x)$ be a clause. Let $H'_3$ be the
Eulerian graph obtained from $G'_3(\{\C\})$ by adding an edge $u^\C_\x m$ in
each of the two copies of the $(\C,\x)$-gadget. The next result follows from the
discussion above.
\begin{prop}
  $H'_3$ is an Eulerian weakly $6$-connected graph that has no $3$-connected
  orientation.
\end{prop}

\section*{Acknowledgement}

I thank Joseph Cheriyan for inviting me to the University of Waterloo and for
his profitable discussions. I thank Zolt\'an Szigeti who observed that the graph
$H_3$ obtained from $G_3$ by complete splitting-off on $t_a$ and on $t_b$ is a
smaller counterexample.  I thank both of them and Abbas Mehrabian for careful
reading of the manuscript.

\bibliographystyle{plain}
\bibliography{../../biblio/biblio}

\begin{thebibliography}{10}

\bibitem{BergJordan2006}
A.~R. Berg and T.~Jordán.
\newblock Two-connected orientations of {E}ulerian graphs.
\newblock {\em Journal of Graph Theory}, 52(3):230--242, 2006.

\bibitem{CheriyanDdGSzigeti}
J.~Cheriyan, O.~Durand~de Gevigney, and Z.~Szigeti.
\newblock Packing of rigid spanning subgraphs and spanning trees.
\newblock Submitted to \emph{Journal of Combinatorial Theory, Series B}.

\bibitem{FordFulkerson1962}
L.~R. Ford and D.~R. Fulkerson.
\newblock {\em Flows in Networks}.
\newblock Princeton Univ. Press, 1962.

\bibitem{Frank1995}
A.~Frank.
\newblock Connectivity and network flows.
\newblock In {\em Handbook of Combinatorics}, pages 111--177. MIT Press, 1995.

\bibitem{Jordan2005}
T.~Jord{\'a}n.
\newblock On the existence of k edge-disjoint 2-connected spanning subgraphs.
\newblock {\em Journal of Combinatorial Theory, Series B}, 95(2):257--262,
  2005.

\bibitem{Kaneko_Ota2000}
A.~Kaneko and K.~Ota.
\newblock On minimally $(n,\lambda)$-connected graphs.
\newblock {\em Journal of Combinatorial Theory, Series B}, 80(1):156 -- 171,
  2000.

\bibitem{KiralySzigeti2006}
Z.~Kir\'{a}ly and Z.~Szigeti.
\newblock Simultaneous well-balanced orientations of graphs.
\newblock {\em Journal of Combinatorial Theory, Series B}, 96(5):684--692,
  2006.

\bibitem{Menger1927}
K.~Menger.
\newblock Zur allgemeinen kurventheorie.
\newblock {\em Fundamental Mathematics}, pages 96--115, 1927.

\bibitem{NashWilliams1960}
C.~St. J.~A. Nash-Williams.
\newblock On orientations, connectivity and odd-vertex-pairings in finite
  graphs.
\newblock {\em Canadian Journal Mathematics}, pages 555--567, 1960.

\bibitem{Robbins1939}
H.~E. Robbins.
\newblock A theorem on graphs, with an application to a problem of traffic
  control.
\newblock {\em The American Mathematical Monthly}, 46(5):pp. 281--283, 1939.

\bibitem{Schaefer1978}
T.~J. Schaefer.
\newblock {The complexity of satisfiability problems}.
\newblock In {\em ACM Symposium on Theory of Computing}, pages 216--226, 1978.

\bibitem{Thomassen1989}
C.~Thomassen.
\newblock Configurations in graphs of large minimum degree, connectivity, or
  chromatic number.
\newblock {\em Annals of the New York Academy of Sciences}, 555(1):402--412,
  1989.

\end{thebibliography}

\end{document}